\documentclass[reqno, 12pt]{amsart}
\usepackage{amsmath,mathtools}
 \usepackage{amssymb}
\usepackage{amsthm}
\usepackage{times}
\usepackage{latexsym}
\usepackage[mathscr]{eucal}

\numberwithin{equation}{section}
 
  \newtheorem{theorem}{Theorem}[section]
  \newtheorem{proposition}[theorem]{Proposition}
  \newtheorem{lemma}[theorem]{Lemma}
  \newtheorem{corollary}[theorem]{Corollary}

  \newtheorem{definition}[theorem]{Definition}
  \newtheorem{example}[theorem]{Example}

\title[A note on quasi generalized CR-lightlike geometry ]{A note on quasi generalized CR-lightlike geometry in indefinite nearly $\mu$-Sasakian manifold}

\author[Fortun\'{e} Massamba, Samuel Ssekajja ]{Fortun\'{e} Massamba*, Samuel Ssekajja**}

\newcommand{\acr}{\newline\indent}

\address{\llap{*\,} School of Mathematics, Statistics and Computer Science\acr
 University of KwaZulu-Natal\acr
 Private Bag X01, Scottsville 3209\acr
South Africa}
\email{massfort@yahoo.fr, Massamba@ukzn.ac.za} 
\thanks{}

\address{\llap{**\,} School of Mathematics, Statistics and Computer Science\acr
 University of KwaZulu-Natal\acr
 Private Bag X01, Scottsville 3209\acr
South Africa}
\email{ssekajja.samuel.buwaga@aims-senegal.org} 
\thanks{}

\subjclass[2010]{Primary 53C25; Secondary 53C40, 53C50}

\keywords{QGCR-lightlike submanifold; Nearly cosympletic manifold; minimal submanifold; ascreen and co-screen submanifolds.}

\begin{document}
 
\begin{abstract} 
The concept of quasi generalized CR-lightlike was first introduced by the authors in \cite{ms}. In this paper, we focus on ascreen and co-screen quasi generalized CR-lightlike submanifolds of indefinite nearly $\mu$-Sasakian manifold. We prove an existence theorem for minimal ascreen quasi generalized CR-lightlike submanifolds admitting a metric connection. Classification theorems on nearly parallel and auto-parallel distributions on a co-screen quasi generalized CR-lightlike submanifold are also given. Several examples are also constructed, where necessary, to illustrate the main ideas.
\end{abstract}

\maketitle

\section{Introduction}  

  Lightlike geometry is widely applied in mathematical physics, particularly in general relativity and electromagnetism (see \cite{db,ds2} and references therein). The theory of lightlike submanifolds was introduced, in 1996, by Bejancu-Duggal \cite{db} and later adopted by many researchers, for instance \cite{dj}, \cite{ds2}, \cite{ds3}, \cite{ds4}, \cite{ma1}, \cite{ma2} and \cite{ma22}. In \cite{ds3}, Duggal-Sahin studied generalized CR (GCR) lightlike submanifolds of an indefinite Sasakian manifold. In their paper, the structure vector field, $\xi$, of the almost contact structure $(\overline{\phi},\xi,\eta)$ was assumed to be tangent to the GCR-lightlike submanifold. Moreover, when $\xi$ is tangent to the submanifold in Sasakian case, Calin \cite{ca} proved that it belongs to the screen distribution. This assumption is widely accepted and it has been applied in many papers on contact lightlike geometry (see \cite{ds1}, \cite{ds2}, \cite{ds3}, \cite{ma1} and other references therein). 
  
  One of the motivations for choosing $\xi$ to be tangent to the CR-submanifold, by many authors, is the  book \cite[p. 48]{yano} by Yano-Kon, in which they showed that any Riemannian submanifold of Sasakian manifold which is normal to the structure vector field, $\xi$, is actually anti-invariant. More precisely, if $A_{\xi}$ (with $\xi$ in the normal bundle) and $g$ denotes the shape operator and the induced metric tensor, respectively, of a Riemannian submanifold of a Sasakian manifold, then one can easily verify that $g(A_{\xi}X,Y)=g(\overline{\phi}X,Y)$, for any tangent vector fields $X$ and $Y$. Since $A_{\xi}$ is symmetric while $\overline{\phi}$ is anti-symmetric with respect to $g$, then one gets $A_{\xi}=0$ and consequently $\overline{\phi}X$ is a normal vector field for any tangent vector $X$. This indicates that the tangent bundle of such submanifold is carried into its normal bundle under the action of $\overline{\phi}$. 
  
  On contrary, the normal bundle of a lightlike submanifold intersect its tangent bundle and, also, its shape operators  are  generally not symmetric with respect to the induced metric tensor and thus one may, if needed, consider some \textit{special} non-tangential almost contact CR-lightlike submanifolds. 
  
 Restricting $\xi$ to the screen distribution comes with immediate benefits, including simplicity and maintenance of the classical contact CR-structures as introduced in Riemannian CR-submanifolds. However, one drawback with this choice is the narrowing of research to only such CR-submanifolds and leaving out other special cases which one gets if $\xi$ is kept general. In a different view,  we introduced a new class of CR-lightlike submanifold of a nearly Sasakian and nearly cosymplectic manifold, known as \textit{quasi generalised CR (QGCR)} lightlike submanifold (see \cite{ms} for details) in which $\xi$ was not necessarily tangent to the submanifold. We proved that any QGCR-lightlike submanifold tangent to the structure vector field is, in fact, a GCR-lightlike submanifold \cite{ds3}. It is important to note that generalizing $\xi$ gives a submanifold (for instance, QGCR-lightlike submanifold) which is lower in dimension to a comparable GCR-lightlike submanifold. For instance, one requires a 7-dimensional submanifold of an 11-dimensional almost contact ambient manifold in order to achieve a 3-lightlike proper ascreen QGCR-lightlike submanifold (see  \cite{ms} for other details), while a 3-lightlike proper GCR submanifold is 9-dimensional and comes from a 13-dimensional almost contact ambient manifold (see \cite{ds3}). Further still, under similar conditions, a 13-dimensional almost contact ambient manifold gives an 8-dimensional 3-lightlike co-screen QGCR submanifold and a 9-dimensional 3-lightlike GCR submanifold (see Example \ref{exa11}).  

The aim of this paper is to study minimal ascreen and co-screen QGCR-lightlike submanifolds of indefinite nearly $\mu$-Sasakian manifold. We prove an existence theorem for irrotational minimal ascreen QGCR-lightlike submanifold, admitting a metric connection and also discuss the geometry of distributions on co-screen QGCR-lightlike submanifold. Also, several examples have been constructed to illustrate the ideas. 

The paper is organized as follows. In Section \ref{prel}, we present the basic notions of nearly $\mu$-Sasakian manifolds and lightlike submanifolds which we refer to in the remaining sections. For extended details on nearly Sasakian and nearly $\mu$-Sasakian manifolds refer to \cite{mo}, \cite{bl1}, \cite{bl2}, \cite{endo} and \cite{diaz}.  In Section \ref{QGCR}, we review the basic notions of QGCR-lightlike submanifolds. In Section \ref{minimal}, we study minimal QGCR-lightlike submanifolds. Finally, in section \ref{Co} we introduce co-screen QGCR-lightlike submanifolds and study the geometry of its distributions.

\section{Preliminaries}\label{prel}
 
Let $\overline{M}^{2n+1}$ be an odd dimensional manifold equipped with an almost contact structure $(\overline{\phi}, \xi, \eta)$, i.e., $\overline{\phi}$ is a tensor field of type $(1, 1)$, $\xi$ is a vector field, and $\eta$ is a 1-form satisfying
\begin{equation}\label{equa1}
\overline{\phi}^{2} = -\mathbb{I} + \eta \otimes\xi,\;\;\eta(\xi)= 1 ,\;\;\eta\circ\overline{\phi} =
0 \;\;\mbox{and}\;\;\overline{\phi}(\xi) = 0.
\end{equation}
Then $(\overline{\phi}, \xi, \eta,\,\overline{g})$ is called an indefinite almost contact metric structure on $\overline{M}$ if $(\overline{\phi}, \xi, \eta)$ is an almost contact structure on $\overline{M}$ and $\overline{g}$ is a semi-Riemannian metric on $\overline{M}$ such that \cite{bl2}, for any vector fields $\overline{X}$, $\overline{Y}$ on  $\overline{M}$,
\begin{equation}\label{equa2}
 \overline{g}(\overline{\phi}\,\overline{X}, \overline{\phi}\,\overline{Y}) = \overline{g}(\overline{X}, \overline{Y}) -  \eta(\overline{X})\,\eta(\overline{Y}),\;\;\mbox{and}\;\;\eta(\overline{X}) =  \overline{g}(\xi,\overline{X}).
\end{equation}
Let $\overline{\nabla}$ denote the Levi-Civita connection on $\overline{M}$ for the semi-Riemannian metric $\overline{g}$. Then, $\overline{M}$ is called an \textit{indefinite nearly $\mu$-Sasakian manifold} if it satisfies
\begin{equation}\label{eqz}
         (\overline{\nabla}_{\overline{X}} \overline{\phi})\overline{Y}+(\overline{\nabla}_{\overline{Y}} \overline{\phi})\overline{X}=\mu\{2\overline{g}(\overline{X}, \overline{Y})\xi-\eta(\overline{Y})\overline{X}-\eta(\overline{X})\overline{Y}\},
    \end{equation}
for any vector fields $\overline{X}$, $\overline{Y}$ on $\overline{M}$.  Notice that when $\mu=0$ (resp. $\mu=1$) then $\overline{M}$ reduces to the known nearly cosymplectic (resp. nearly Sasakian) manifold.

Through out this paper, $\Gamma(\Xi)$ will denote the set of smooth sections of the vector bundle $\Xi$. 

By letting  $\overline{Y}=\xi$ in (\ref{eqz}) we obtain
   \begin{equation}
    \overline{\nabla}_{\overline{X}} \xi+\overline{\phi}(\overline{\nabla}_\xi \overline{\phi})\overline{X} =-\mu\overline{\phi}\,\overline{X},\label{eq10}
   \end{equation}
for any $\overline{X}\in \Gamma (T\overline{M})$. Let $\overline{H}$ be a (1,1)-tensor  on $\overline{M}$ taking
\begin{equation*}
 \overline{H}\,\overline{X}=\overline{\phi}(\overline{\nabla}_\xi \overline{\phi})\overline{X}, 
\end{equation*}
for any $\overline{X}\in \Gamma (T\overline{M})$, such that (\ref{eq10}) reduces to 
$
    \overline{\nabla}_{\overline{X}} \xi =-\mu\overline{\phi}\, \overline{X}-\overline{H}\,\overline{X}.
$

By a straightforward calculation, one can show that the linear operator $\overline{H}$ satisfies the following properties
 \begin{align}
  & \overline{H}\,\overline{\phi} + \overline{\phi}\,\overline{H}=0,\;\;\overline{H}\xi=0,\;\;\eta\circ \overline{H}=0, \nonumber\\
  \mbox{and}\;\;& \overline{g}(\overline{H}\,\overline{X}, \overline{Y})=-\overline{g}(\overline{X}, \overline{H}\,\overline{Y})\;\;\;\; ( \overline{H} \;\;\mbox{is skew-symmetric}).
 \end{align}
Moreover,  $\overline{M}$ is $\mu$-Sasakian if and only if $\overline{H}$ vanishes identically on $\overline{M}$ (see \cite{bl2}).

Let $(\overline{M}^{m+n}$, be semi-Riemannian manifold of constant index $\nu$, $1\le \nu\le m+n$ and $M$ be a submanifold of $\overline{M}^{m}$ of codimension $n$, both $m,n\ge 1$. For any point $p\in M$, the orthogonal complement $T_{p} M^{\perp}$ of the tangent space $T_{p} M$ is given by
 $$
 T_{p} M^{\perp} = \{X\in T_{p} M: \overline{g}(X, Y)=0,\; \forall \, Y\in T_{p} M\}.
 $$
 Let  $\mathrm{Rad} \, T_{p} M = \mathrm{Rad}\, T_{p} M^{\perp} = T_{p} M \cap T_{p} M^{\perp}$. The submanifold $M$ of $\overline{M}$ is said to be $r$-lightlike submanifold (one supposes that the index of $\overline{M}$ is $q \ge r$), if the mapping 
 $
 \mathrm{Rad} \, T M: p\in M \longrightarrow\mathrm{Rad}\, T_{p} M 
 $
 defines a smooth distribution on $M$ of rank $r > 0$. We call $\mathrm{Rad}\,T M$ the radical distribution on $M$. In the sequel, an $r$-lightlike submanifold will simply be called a \textit{lightlike submanifold} and $g$ is \textit{lightlike metric}, unless we need to specify $r$.
 
 Let $S(T M)$ be a screen distribution which is a semi-Riemannian complementary distribution of $\mathrm{Rad}\,T M$ in $T M$, that is,
 \begin{equation}\label{eq05}
  T M = \mathrm{Rad}\,T M \perp S(T M).
 \end{equation} 
 Let us consider a screen transversal bundle $S(TM^\perp)$, which is semi-Riemannian and complementary to $\mathrm{Rad}\, TM$ in $TM^\perp$. Since, for any local basis $\{E_i \}$ of  $\mathrm{Rad}\,TM$, there exists a local null frame $\{N_i\}$ of sections with values in the orthogonal complement of $S(T M^\perp)$ in $S(T M )^\perp$  such that $g(E_i , N_j ) = \delta_{ij}$ and $\overline{g}(N_{i},N_{j})=0$. It follows that there exists a lightlike transversal vector bundle $l\mathrm{tr}(TM)$ locally spanned by $\{N_i\}$ \cite{db}. 

Let $\mathrm{tr}(TM)$ be complementary (but not orthogonal) vector bundle to $TM$ in $T\overline{M}$. Then, 
\begin{align}
          \mathrm{tr}(TM)  & =l\mathrm{tr}(TM)\perp S(TM^\perp),\label{eq08}\\
  T\overline{M} &=  S(TM)\perp S(TM^\perp)\perp\{\mathrm{Rad}\, TM\oplus l\mathrm{tr}(TM)\}\label{eq04} .
\end{align}
Note that the distribution $S(TM)$ is not unique, and is canonically isomorphic to the factor vector bundle $TM/ \mathrm{Rad}\, TM$ given by Kupeli \cite{kup}. 
 
The following classification of  a lightlike submanifold $M$ of $\overline{M}$ are well-known \cite{db}: i). $M$ is $r$-null if $1\leq r< min\{m,n\}$; ii). $M$ is co-isotropic if $1\leq r=n<m$,  $S(TM^\perp)=\{0\}$; iii). $M$ is isotropic if $1\leq r=m<n$,  $S(TM)=\{0\}$; iv). $M$ is totally lightlike if $r=n=m$,  $S(TM)=S(TM^\perp)=\{0\}$.  

Consider a local quasi-orthonormal fields of frames of $\overline{M}$ along $M$ as 
\begin{equation*}
\{ E_1,\cdots, E_r,N_1,\cdots, N_r,X_{r+1},\cdots,X_{m},W_{1+r},\cdots, W_{n}\},
\end{equation*}
where
 $\{X_{r+1},\cdots,X_{m}\}$ and $\{W_{1+r},\ldots, W_n\}$ are respectively orthogonal bases of $\Gamma(S(TM)|_{U})$ and $\Gamma(S(TM^{\perp})|_{U})$  and that $\epsilon_l= g(X_{l},X_{l})$ and $\epsilon_\alpha=\overline{g}(W_\alpha,W_\alpha)$ be the signatures of $X_{l}$ and $W_\alpha$ respectively.
 
Let $P$ be the projection morphism of $TM$ on to $S(TM)$.  the following  Gauss-Weingartein equations  of an $r$-lightlike submanifold $M$ and $S(TM)$ are well-known \cite{ds2};
  \begin{align}
     & \overline{\nabla}_X Y=\nabla_X Y+\sum_{i=1}^r h_i^l(X,Y)N_i+\sum_{\alpha=r+1}^n h_\alpha^s(X,Y)W_\alpha,\label{eq11}\\
     & \overline{\nabla}_X N_i=-A_{N_i} X+\sum_{j=1}^r \tau_{ij}(X) N_j+\sum_{\alpha=r+1}^n \rho_{i\alpha}(X)W_\alpha,\label{eq31}\\
     & \overline{\nabla}_X W_\alpha=-A_{W_\alpha} X+\sum_{i=1}^r \varphi_{\alpha i}(X) N_i+\sum_{\beta=r+1}^n \sigma_{\alpha\beta}(X)W_\beta,\label{eq32}\\
     & \nabla_X P Y=\nabla_X^* PY+\sum_{i=1}^r h_i^*(X, P Y)E_i,\\
     & \nabla_X E_i=-A_{E_i}^* X-\sum_{j=1}^r \tau_{ji}(X) E_j,\;\;\;\; \forall\; X,Y\in \Gamma(TM)\label{eq50},
  \end{align}
where  $\nabla$ and $\nabla^*$ are the induced connections on $TM$ and $S(TM)$ respectively. Further, $h_i^l$ and $h_\alpha^s$ are symmetric bilinear forms known as \textit{local lightlike} and \textit{screen fundamental} forms of $TM$ respectively. Also $h_i^*$ are the \textit{ second fundamental forms} of $S(TM)$. $A_{N_i}$, $A_{E_i}^*$ and $A_{W_\alpha}$ are linear operators on $TM$ while $\tau_{ij}$, $\rho_{i\alpha}$, $\varphi_{\alpha i}$ and $\sigma_{\alpha\beta}$ are 1-forms on $TM$.  

Notice from (\ref{eq11}) that the second fundamental form $h$  of $M$ is given by
 \begin{equation}\label{h1}
 h(X,Y)=\sum_{i=1}^r h_i^l(X,Y)N_i+\sum_{\alpha=r+1}^n h_\alpha^s(X,Y)W_\alpha,
\end{equation}
for any $X,Y\in \Gamma(TM)$.  It is easy to see that  $\nabla^*$ is a metric connection on $S(TM)$, while $\nabla$ is generally not a metric connection and satisfies:
      \begin{equation}\label{metric}
         (\nabla_X g)(Y,Z)=\sum_{i=1}^r\{h_i^l(X,Y)\lambda_i(Z)+h_i^l(X,Z)\lambda_i(Y)\},
      \end{equation}
for any $X,Y\in \Gamma(TM)$ and $\lambda_i$ are 1-forms given by $\lambda_i(\cdot)=\overline{g}(\cdot,N_i)$.
     
The above three local second fundamental forms are related to their shape operators by the following equations.
      \begin{align} 
           g(A_{E_i}^*X,Y)&=h_i^l(X,Y)+\sum_{j=1}^rh_j^l(X,E_i)\lambda_j(Y), \;\; \overline{g}(A_{E_i}^*X,N_j)=0,\nonumber \\
          g(A_{W_\alpha}X,Y) &=\epsilon_\alpha h_\alpha^s(X,Y)+ \sum_{i=1}^r \varphi_{\alpha i}(X)\lambda_i(Y),\nonumber 
          \end{align}
           \begin{align}
          \overline{g}(A_{W_\alpha}X,N_i)& =\epsilon_\alpha \rho_{i\alpha}(X), \nonumber\\
          g(A_{N_i}X,Y)& =h_i^*(X,\mathcal{P}Y),\;\;\; \lambda_j(A_{N_i}X)+\lambda_i(A_{N_j}X)=0,\nonumber
      \end{align}
for any $X,Y\in \Gamma(TM)$.

\section{Quasi generalized CR-lightlike submanifolds}\label{QGCR}

The structure vector field $\xi$ is globally defined on $T\overline{M}$. Therefore, one can define it according to decomposition (\ref{eq08}) as follows;
         \begin{equation}\label{eq2}
              \xi=\xi_{S}+\sum_{i=1}^ra_i E_i+\sum_{i=1}^rb_i N_i+\sum_{\alpha=r+1}^nc_\alpha W_\alpha,
         \end{equation}
where $\xi_{S}$ is a smooth vector field of $S(TM)$ while $a_i=\eta(N_i)$, $b_i=\eta(E_i)$ and $c_\alpha=\epsilon_\alpha\eta(W_\alpha)$ all smooth functions on $\overline{M}$.

Now, we adopt the definition of quasi generalized CR (QGCR)-lightlike submanifolds given in \cite{ms} for indefinite nearly $\mu$-Sasakian manifolds.
\begin{definition} \label{def2}{\rm Let $(M,g,S(TM),S(TM^\perp))$  be a lightlike submanifold of an indefinite nearly $\mu$-Sasakian manifold $(\overline{M}, \overline{g})$. We say that $M$ is quasi generalized CR (QGCR)-lightlike submanifold of $\overline{M}$ if the following conditions are satisfied:
\begin{enumerate}
 \item [(i)] there exist two distributions $D_1$ and $D_2$ of $\textrm{Rad}\,TM$  such that 
        \begin{equation}\label{eq03}
             \mathrm{Rad}\, TM = D_1\oplus D_2, \;\;\overline{\phi} D_1=D_1, \;\;\overline{\phi} D_2\subset S(TM),
        \end{equation}
 \item [(ii)] there exist vector bundles $D_0$ and $\overline{D}$ over $S(TM)$ such that 
         \begin{align} 
              & S(TM)=\{\overline{\phi} D_2 \oplus \overline{D}\}\perp D_0,\\ 
              \mbox{with}\;\;\; &\overline{\phi}D_{0} \subseteq D_{0},\;\;   \overline{D}= \overline{\phi} \, \mathcal{S}\oplus \overline{\phi} \,\mathcal{L}, \label{s81}
         \end{align}    
\end{enumerate}
where $D_0$ is a non-degenerate distribution on $M$, $\mathcal{L}$ and $\mathcal{S}$ are respectively vector subbundles of $l\mathrm{tr}(TM)$  and $S(TM^{\perp})$.
}
\end{definition}

If $D_{1}\neq \{0\}$, $D_0\neq \{0\}$, $D_2\neq \{0\}$ and $\mathcal{S}\neq \{0\}$, then $M$ is called a \textit{proper QGCR-lightlike submanifold}. Through out this paper we shall suppose that $M$ is a proper QGCR-lightlike submanifold.

\begin{proposition}
 A QGCR-lightlike submanifold $M$ of an indefinite nearly $\mu$-Sasakian manifold $\overline{M}$ tangent to the structure vector field $\xi$ is a GCR-lightlike submanifold.
 \end{proposition} 
 \begin{proof}
The proof uses similar arguments as in \cite{ms}.
 \end{proof}
With reference to  (\ref{eq05}), $TM$ can be rewritten as
\begin{equation}
 TM  =D \oplus \widehat{D}, \;\; \mbox{with}\;\;D =D_0\perp D_{1}\;\mbox{and}\;\widehat{D}=\{D_{2}\perp\overline{\phi} D_{2}\} \oplus\overline{D}.\nonumber
\end{equation} 
Notice that $D$ is invariant with respect to $\overline{\phi}$ while $\widehat{D}$ is not generally anti-invariant with respect to $\overline{\phi}$ as it is the case with the classical GCR-lightlike submanifolds \cite{ds3}.

Let  $(M,g,S(TM),S(TM^\perp))$ be a QGCR-lightlike submanifold  of $\overline{M}$, then from  Definition \ref{def2}:
          \begin{enumerate}
               \item condition (i) implies that $\dim(\mathrm{Rad}\, TM)=s\ge  3$,
               \item condition (ii) implies that $\dim(D)\ge  4l\ge 4$ and $\dim(D_2)= \dim(\mathcal{L})$. 
          \end{enumerate}       

\section{Minimal ascreen QGCR-lightlike submanifolds}\label{minimal} 
Consider a quasi-orthonormal frame a long $T\overline{M}$ given by 
\begin{equation}\label{cons}
\{E_{1},\cdots, E_{r} , X_{1} ,\cdots, X_{m}, W_{1},\cdots,W_{n}, N_{1},\cdots, N_{r}\},
 \end{equation}
 such that $\{E_{1},\cdots, E_{q} , X_{1} ,\cdots, X_{m}\}\in TM$. Let us suppose that $\{E_{1},\cdots, E_{2p}\}$, $\{E_{2p+1},\cdots, E_{r}\}$ and $\{X_{1} ,\cdots, X_{2l}\}$ are respectively the  bases of $D_{1}$, $D_{2}$ and $D_{0}$. Further, let $\{W_{r+1},\cdots,W_{k}\}$ and $\{N_{2p+1},\cdots,N_{r}\}$, respectively be the bases of $\mathcal{S}$ and $\mathcal{L}$.

\begin{definition}[\cite{dh1}]\label{def3}  {\rm
A lightlike submanifold $M$ of a semi-Riemannian manifold $\overline{M}$ is said to be \textit{ascreen} if the structure vector field, $\xi$, belongs to $\mathrm{Rad}\, TM \oplus l\mathrm{tr}(T M)$.
}
\end{definition}
The following result for ascreen QGCR-lightlike submanifolds is well-known (see Lemma 3.6 and Theorem 3.7 of \cite{ms}).

\begin{theorem}\label{asc}
 Let  $(M,g,S(TM),S(TM^\perp))$ be an ascreen QGCR-lightlike submanifold of an indefinite nearly $\mu$-Sasakian manifold $\overline{M}$, then $\xi\in\Gamma(D_{2}\oplus\mathcal{L})$. Further, if $M$ is a 3-lightlike QGCR submanifold, then $M$ is ascreen lightlike submanifold if and only if $\overline{\phi}\mathcal{L}=\overline{\phi} D_{2}$.
\end{theorem} 
 From (\ref{cons}) and  (\ref{eq2}), we can write the generalized structure vector field of an ascreen QGCR-lightlike submanifold as  
 \begin{equation}\label{T1}
  \xi=\sum_{i=2p+1}^{r}a_{i}E_{i}+\sum_{i=2p+1}^{r}b_{i}N_{i},
 \end{equation}
where $a_{i}=\overline{g}(N_{i},\xi)$ and $b_{i}=\overline{g}(E_{i},\xi).$
 
 Now, using (\ref{cons}) and Theorem \ref{asc} above, we deduce the following for an $r$-lightlike ascreen QGCR-submanifold $(M,g)$.
 \begin{proposition}\label{prop}
 Let $M$ be a proper $r$-lightlike ascreen QGCR submanifold, where $r\ge 3$, of an indefinite almost contact manifold $\overline{M}$. Then, there exist at least one  pair $\{E_{u},N_{u}\}\subset \mathcal{L}\oplus D_{2}$ and a corresponding non-vanishing real valued smooth function $\sigma_{u}$, where $u\in\{2p+1,\cdots,r\}$, such that $\overline{\phi}N_{u}=\sigma_{u}\overline{\phi}E_{u}$ and $\dim( \overline{\phi}\mathcal{L}\oplus  \overline{\phi}D_{2})\ge1$. Equality occurs when $r=3$.
\end{proposition}
\begin{proof}
The proof follows from Theorem \ref{asc} above.
\end{proof}
As an example, we construct a $4$-lightlike ascreen QGCR submanifold. Let us consider the case $\mu=0$ and $\overline{H}=0$. That is,  $\overline{M}=(\mathbb{R}_{q}^{2m+1}, \overline{\phi}_{0} ,\xi, \eta, \overline{g})$ is an indefinite cosymplectic manifold with the usual cosympletic structure;
\begin{align*}
   &\eta = dz,\quad\xi=\partial z,\\
   &\overline{g} =\eta\otimes\eta-\sum_{i=1}^{\frac{q}{2}}(dx^i\otimes dx^i+dy^i\otimes dy^i)+\sum_{i=q+1}^{m}(dx^i\otimes dx^i+dy^i\otimes dy^i),\\
   &\phi_0 (\sum_{i=1}^m(X_i\partial x^i+Y_i\partial y^i)+Z\partial z )=\sum_{i=1}^m(Y_i\partial x^i-X_i\partial y^i),
\end{align*}
where $(x^{i} , y^{i} , z)$ are Cartesian coordinates and $\partial t_{k} = \frac{\partial}{\partial t^{k}}$, for $t\in\mathbb{R}^{2m+1}$.

\begin{example}
 {\rm
 Let $\overline{M}=(\mathbb{R}_{6}^{15}, \overline{g})$ be a semi-Euclidean space,  where $\overline{g}$ is of signature $(-,-,-,+,+,+,+, -,-,-,+,+,+,+,+)$ with respect to the canonical basis
\begin{equation*}
 (\partial x_{1},\partial x_{2},\partial x_{3},\partial x_{4},\partial x_{5},\partial x_{6},\partial x_{7},\ \partial y_{1},\partial y_{2},\partial y_{3},\partial y_{4},\partial y_{5},\partial y_{6},\partial y_{7},\partial z).
\end{equation*}
Let $(M,g)$ be a submanifold of $\overline{M}$ given by 
\begin{align*}
 y^1=-x^4,\;\; y^{2}=x^{5},\;\; y^{3}=\sqrt{z^{2}- (x^{3})^{2}},\;\;  y^4=x^1,\;\; y^{6}=x^{6},\;x^{3},y^{3}>0.
 \end{align*}
By direct calculations, one can easily check that the vector fields
 \begin{align*}
  E_{1}& =\partial x_{4}+\partial y_{1},\quad E_{2}=\partial x_{1}-\partial y_{4},\quad E_{3}=\partial x_{5}+\partial y_{2},\\
  E_{4}&=x^{3}\partial x_{3}+y^{3}\partial y_{3}+z\partial z,\quad X_{1}=\partial x_{6}+\partial y_{6},\quad X_{2}=\partial x_{2}-\partial y_{5},\\
  X_{3}&=y^{3}\partial x_{3}-x^{3}\partial y_{3},\quad X_{4}=-\partial x_{2}-\partial y_{5},\quad X_{5}=\partial y_{7}, \quad X_{6}=\partial x_{7}, 
 \end{align*}
form a local frame of $TM$. With reference to  the above frame, we see that  $\mathrm{Rad} \, TM $ is spanned by $\{E_{1}, E_{2}, E_{3}, E_{4}\}$, and therefore $M$ is a 4-lightlike submanifold. Further more, $\overline{\phi}_0 E_1=E_2$, therefore we set $D_{1}=\mbox{span}\{E_{1},E_{2}\}$. Notice that $\overline{\phi}_0 E_3=X_{2}$ and $\overline{\phi}_{0}E_{4}=X_{3}$ thus, $D_2=\mathrm{span}\{E_{3},E_{4}\}$. Also, $\overline{\phi}_{0} X_{5}=X_{6}$, so we set $D_{0}=\mathrm{span}\{X_{5},X_{6}\}$. Further, by following direct calculations, we have 
\begin{align} 
 N_{1} & =\frac{1}{2}(\partial x_{4}-\partial y_1),\quad  N_{2}=\frac{1}{2}(-\partial x_{1}-\partial y_{4}),\quad N_{3}=\frac{1}{2}(\partial x_{5}-\partial y_{2})\nonumber\\
 N_{4} & =\frac{1}{2z^{2}}(- x^{3}\partial x_{3} - y^{3}\partial y_{3}+z\partial z),\quad W=\partial x_{6}-\partial y_{6}.\nonumber 
\end{align}
Note that $l\mathrm{tr}(TM)=\mathrm{span}\{N_{1},N_{2},N_{3},N_{4}\}$ and $\mathcal{S}=\mbox{span}\{W\}$. It is easy to see that $\overline{\phi}_{0} N_{2}=-N_{1}$ and $\overline{\phi}_{0} N_{3}=X_{4}$. Notice $\overline{\phi}_{0} N_{4}=-\frac{1}{2z^{2}}X_{3}=-\frac{1}{2z^{2}}\overline{\phi}E_{4}$ and, hence, $\sigma_{4}=-\frac{1}{2z^{2}}$ (see Proposition \ref{prop}). Therefore, $\mathcal{L}=\mathrm{span}\{N_{3},N_{4}\}$.  Also, $\overline{\phi}W=-X_{1}$ and hence $\mathcal{S}=\mathrm{span}\{W\}$. Observe that $\overline{\phi}\mathcal{L}\oplus  \overline{\phi}D_{2}=\mbox{span}\{X_{2}, X_{3},X_{4}\}$ and therefore $\dim(\overline{\phi}\mathcal{L}\oplus  \overline{\phi}D_{2})=3$. Applying $\overline{\phi}_{0}$ to  (\ref{T1}) and substituting the corresponding $\overline{\phi}_{0}E_{i}$s and $\overline{\phi}_{0}N_{i}$s for $i=3,4$ we obtain $a_{3}+b_{3}=0$, $a_{3}-b_{3}=0$ and $2z^{2}a_{4}=b_{4}$.  Finally, we get $\xi=\frac{1}{2z}E_{4}+zN_{4}$.  Since $\overline{\phi}_0\xi=0$ and $\overline{g}(\xi,\xi)=1$, we see that $(M,g)$ is a 4-lightlike ascreen QGCR submanifold of $\overline{M}$ satisfying the hypothesis of Proposition \ref{prop}.
 }
\end{example}

Next, we adapt the definition of minimal lightlike submanifolds given by \cite{ds3}.
\begin{definition} \label{mini}{\rm
 A lightlike submanifold $(M,g,S(TM))$ of a semi-Riemannian $(\overline{M},\overline{g})$ is called minimal if;
 \begin{enumerate}
  \item  $h^{s}=0$ on $\mathrm{Rad} \, TM$ and,
  \item $\mathrm{trace}(h)=0$, where trace is writen with respect to $g$ restricted to $S(TM)$.
 \end{enumerate}}
\end{definition}
It is well-known that the Definition \ref{mini} is independent of the choice of the screen distribution $S(TM)$ \cite{ds3}.  

Now, we construct a minimal ascreen QGCR-lightlike submanifold, which is note totally geodesic, of a  nearly $\mu$-Sasakian manifold with $\mu=0$ and $\overline{H}=0$ (i.e., the ambient space is a cosymplectic manifold).

\begin{example}\label{exa1}
{\rm
Let $\overline{M}=(\mathbb{R}_4^{13}, \overline{g})$ be a semi-Euclidean space,  where $\overline{g}$ is of signature $(-,-,+,+,+,+, - , -,+,+,+,+,+)$ with respect to the canonical basis
\begin{equation*}
 (\partial x_{1},\partial x_{2},\partial x_{3},\partial x_{4},\partial x_{5},\partial x_{6}, \partial y_{1},\partial y_{2},\partial y_{3},\partial y_{4},\partial y_{5},\partial y_{6},\partial z).
\end{equation*}
Let $(M,g)$ be a submanifold of $\overline{M}$ given by 
\begin{align*}
 &x^{1}=\omega^{1},\;\;x^{2}=\omega^{2},\;\;x^{3}=\omega^{3},\;\;x^{4}=\omega^{4},\;\;x^{5}=\cos\omega^{5}\cosh\omega^{6},\\&x^{6}=\sin\omega^{5}\cosh\omega^{6}, \;\; y^{1}=-\omega^{4},\;\;y^{2}=\sqrt{2}\omega^{8}-\omega^{2},\;\;\;y^{3}=\omega^{7},\\& y^{4}=\omega^{1},\;\;  y^{5}=\cos\omega^{5}\sinh\omega^{6},\;\;y^{6}=\sin\omega^{5}\sinh\omega^{6},\;\;\;\;z=\omega^{8}.
 \end{align*}
By direct calculations, we can easily check that the vector fields
\begin{align*} 
 E_1 & =\partial x_4+\partial y_1,\quad E_2=\partial x_1-\partial y_4,\\
 E_3 & = \partial x_2 +\partial y_2+\sqrt{2}\partial z, \quad X_{1}=-x^{6}\partial x_{5}+x^{5}\partial x_{6}  -y^{6}\partial y_{5}+y^{5}\partial y_{6}\\ 
 X_{2} & =-\partial x_{2} +\partial y_{2},\quad  X_{3}=y^{5}\partial x_{5}+y^{6}\partial x_{6}+x^{5}\partial y_{5}+x^{6}\partial y_{6},\\
 X_{4}&=\partial y_{3},\; X_{5}=\partial x_{3}, 
\end{align*}
form a local frame of $TM$. From the above frame, we can see that $\mathrm{Rad} \, TM$ is spanned by $\{E_{1}, E_{2}, E_{3}\}$, and therefore, $M$ is a 3-lightlike submanifold. Further, $\overline{\phi}_0 E_1=E_2$, therefore we set $D_{1}=\mbox{span}\{E_{1},E_{2}\}$. Also $\overline{\phi}_0 E_3=-X_2$ and thus $D_2=\mathrm{span}\{E_3\}$. It is easy to see that $\overline{\phi}_{0} X_{4}=X_{5}$, so we set $D_{0}=\mathrm{span}\{X_{4},X_{5}\}$. On the other hand, following direct calculations, we have 
\begin{align} 
 N_{1} & =\frac{1}{2}(\partial x_{4}-\partial y_1),\;  N_{2}=\frac{1}{2}(-\partial x_{1}-\partial y_{4}),\nonumber\\
 N_{3} & =\frac{1}{4}(- \partial x_{2} - \partial y_{2}+\sqrt{2}\partial z),\;   W_{1}=-y^{6}\partial x_{5}+y^{5}\partial x_{6}+x^{6}\partial y_{5}-x^{5}\partial y_{6}, \nonumber\\
 W_{2}&=x^{5}\partial x_{5}+x^{6}\partial x_{6}-y^{5}\partial y_{5}-y^{6}\partial y_{6},\nonumber
\end{align}
from which $l\mathrm{tr}(TM)=\mathrm{span}\{N_{1},N_{2},N_{3}\}$ and 
$S(TM^\perp)=\mathrm{span}\{W_{1},W_{2}\}$. Clearly, $\overline{\phi}_{0} N_{2}=-N_{1}$. 
Further, $\overline{\phi}_0 N_{3}=\frac{1}{4} X_{2}$ and thus 
$\mathcal{L}=\mbox{span}\{N_{3}\}$. Notice that 
$\overline{\phi}_{0}N_{3}=-\frac{1}{4}\overline{\phi}_{0} E_{3}$, which implies that $\sigma_{3}=-\frac{1}{4}$ and therefore, 
$\overline{\phi}_0\mathcal{L}=\overline{\phi}_{0} D_{2}$. Also, $\overline{\phi}_{0} 
W_{1}=-X_{1}$ and $\overline{\phi}_{0} 
W_{2}=-X_{3}$. Therefore $\mathcal{S}=\mbox{span}\{W_{1},W_{2}\}$. Now, we calculate 
$\xi$ as follows: Using (\ref{T1}) we have $\xi=a_{3} E_{3}+b_{3} N_{3}$. Applying 
$\overline{\phi}_{0}$ to this equation we obtain $a_{3}\overline{\phi}_{0} 
E_{3}+b_{3}\overline{\phi}_{0} N_{3}=0$. Now, substituting for $\overline{\phi}_{0} E_{3}$ and 
 $\overline{\phi}_{0} N_{3}$ in this equation we get $4a_{3}=b_{3}$, from which we get 
$\xi=\frac{1}{2\sqrt{2}}E_{3}+\sqrt{2}N_{3}$. Since $\overline{\phi}_0\xi=0$ and 
$\overline{g}(\xi,\xi)=1$, we see that $(M,g)$ is a proper ascreen QGCR-lightlike 
submanifold of $\overline{M}$. Finally, we verify the minimality of $(M,g)$. By simple calculations one can verify easily that the following vectors;
\begin{align*}
 & \widehat{X}_{1}=\frac{1}{\sqrt{\varrho}}X_{1},\;\;\; \widehat{X}_{2}=\frac{1}{\sqrt{2}}X_{2},\;\;\widehat{X}_{3}=\frac{1}{\sqrt{\varrho}}X_{3},\;\;\;\widehat{X}_{4}=X_{4},\;\;\widehat{X}_{5}=X_{5},\\
 &\widehat{W}_{1}=\frac{1}{\sqrt{\varrho}}W_{1},\;\;\;\widehat{W}_{2}=\frac{1}{\sqrt{\varrho}}W_{2},\;\;\;\;\mbox{where}\;\;\; \varrho:=\cosh2\omega^{6},
\end{align*}
are unit vector fields. Moreover,  $\epsilon_{2}=g(\widehat{X}_{2},\widehat{X}_{2})=-1$, $\epsilon_{i}=g(\widehat{X}_{i},\widehat{X}_{i})=1$, for $i=1,3,4,5$ and $\epsilon_{\alpha}=\overline{g}(\widehat{W}_{\alpha},\widehat{W}_{\alpha})=1$, for $\alpha=1,2$. Now, applying (\ref{eq11}) and Koszul's formula (see \cite{db}) one gets $ h(E_{1},E_{1}) =h(E_{2},E_{2})=h(E_{3},E_{3})=h(\widehat{X}_{2},\widehat{X}_{2})=0$, $ h(\widehat{X}_{4},\widehat{X}_{4})=h(\widehat{X}_{5},\widehat{X}_{5})=0$, $h^{l}(\widehat{X}_{1},\widehat{X}_{1})=h^{l}(\widehat{X}_{3},\widehat{X}_{3})=0$, $h^{s}(\widehat{X}_{1},\widehat{X}_{1}) =-\frac{1}{\varrho\sqrt{\varrho}}\widehat{W}_{2}$ and $h^{s}(\widehat{X}_{3},\widehat{X}_{3})=\frac{1}{\varrho\sqrt{\varrho}}\widehat{W}_{2}$. Hence, $M$ is not a totally geodesic ascreen QGCR-lightlike submanifold. Also, we have
$
 \mathrm{trace}(h)|_{S(TM)}=\epsilon_{1}h^{s}(\widehat{X}_{1},\widehat{X}_{1})+\epsilon_{2}h^{s}(\widehat{X}_{3},\widehat{X}_{3})=0.
$
Therefore, $M$ is a minimal proper ascreen QGCR-lightlike submanifold of $\overline{M}$.
}
\end{example}
\begin{definition}[\cite{ds2}]
{\rm
 A lightlike submanifold $M$ of a semi-Riemannian manifold $(\overline{M},\overline{g})$ is called irrotational if $\overline{\nabla}_{X}E\in\Gamma(TM)$, for any $E\in\Gamma(\mathrm{Rad} \, TM))$ and $X\in\Gamma(TM)$. Equivalently, $M$ is irrotational if 
\begin{equation}\label{ms40}
 h^{l}(X,E)= h^{s}(X,E)=0, \;\forall \; X\in\Gamma(TM),\; E\in\Gamma(\mathrm{Rad} \,TM).
\end{equation} 
}
\end{definition}

\begin{theorem}
 Let $(M,g)$ be an irrotational ascreen QGCR-lightlike submanifold of an indefinite nearly $\mu$-Sasakian manifold $(\overline{M},\overline{g})$. If  $\nabla$ is a metric connection, then $M$ is minimal if $\mathrm{trace}(A_{W_{\alpha}})|_{S(TM)}=0.$
\end{theorem}
\begin{proof}
 First, we notice that when $M$ is irrotational then (\ref{ms40}) implies that $h^{s}=0$ on $\mathrm{Rad} \,TM$. Thus, condition (i) of Definition \ref{mini} is satisfied.  Now using Definition \ref{def2}  we can see that the screen distribution $S(TM)$ is generally spanned by 
 \begin{equation}\label{K8}
 \{X_{1},\cdots,X_{2l},\overline{\phi}E_{2p+1},\cdots,\overline{\phi}E_{r},\overline{\phi}N_{2p+1},\cdots,\overline{\phi}N_{r}, \overline{\phi}W_{r+1}, \cdots,\overline{\phi}W_{k}\}.
 \end{equation}
  Since $M$ is ascreen QGCR-lightlike submanifold, the dimension of the frame in (\ref{K8}) is lower than that of a comparable GCR-lightlike subamanifold due to existence of some $u\in\{2p+1,\cdots,r\}$ and non-vanishing smooth function(s) $\sigma_{u}$ such that $\overline{\phi}N_{u}=\sigma_{u}\overline{\phi}E_{u}$ (see Proposition \ref{prop} above). Furthermore, the vectors $\overline{\phi}E_{u}$ and $\overline{\phi}N_{u}$ are non-null, since $g(\overline{\phi}E_{u},\overline{\phi}E_{u})=-a_{u}b_{u}\neq0$ and $g(\overline{\phi}N_{u},\overline{\phi}N_{u})=-a_{u}b_{u}\neq0$. Thus, by setting $Z_{u}=\overline{\phi}N_{u}=\sigma_{u}\overline{\phi}E_{u}$ we have
 \begin{align}\label{K1}
  &\mathrm{trace}(h)|_{S(TM)}=\sum_{t=1}^{2l}\epsilon_{t}h(X_{t},X_{t})+\sum_{j=2p+1}^{\kappa}h(\overline{\phi}E_{j},\overline{\phi}E_{j})\nonumber\\
   &\sum_{j=2p+1}^{\kappa}h(\overline{\phi}N_{j},\overline{\phi}N_{j})+ \sum_{u=\kappa+1}^{r}\epsilon_{u}h(Z_{u},Z_{u})+\sum_{d=r+1}^{k}\epsilon_{d}h(\overline{\phi}W_{d},\overline{\phi}W_{d}).
 \end{align}
Applying (\ref{h1}) to (\ref{K1}) and replacing $Z$ with $E_{i}$ in (\ref{metric}) we derive 
\begin{align}\label{K3}
         (\nabla_X g)(Y,E_{i})=\sum_{i=1}^{r} h_i^l(X,Y)\lambda_i(E_{i})=\overline{g}(h^{l}(X,Y),E_{i}),
\end{align}
for any $X,Y\in\Gamma(S(TM))$.
Then using (\ref{K3}) and the assumption $\nabla$ is a metric connection we get
 \begin{align}\label{K4}
  \mathrm{trace}(h)|_{S(TM)}&=\sum_{t=1}^{2l}\frac{\epsilon_{t}}{n}\sum_{\alpha=r+1}^{n}\epsilon_{\alpha}\overline{g}(h^{s}(X_{t},X_{t}),W_{\alpha})W_{\alpha}\nonumber\\
  &+\sum_{j=2p+1}^{\kappa}\frac{1}{n}\sum_{\alpha=r+1}^{n}\epsilon_{\alpha}\overline{g}(h^{s}(\overline{\phi}E_{j},\overline{\phi}E_{j}),W_{\alpha})W_{\alpha}\nonumber\\
  &+\sum_{j=2p+1}^{\kappa}\frac{1}{n}\sum_{\alpha=r+1}^{n}\epsilon_{\alpha}\overline{g}(h^{s}(\overline{\phi}N_{j},\overline{\phi}N_{j}),W_{\alpha})W_{\alpha}\nonumber\\
  &+\sum_{d=r+1}^{k}\frac{\epsilon_{d}}{n}\sum_{\alpha=r+1}^{n}\epsilon_{\alpha}\overline{g}(h^{s}(\overline{\phi}W_{d},\overline{\phi}W_{d}),W_{\alpha})W_{\alpha}\nonumber\\
  &+\sum_{u=\kappa+1}^{r}\frac{\epsilon_{u}}{n}\sum_{\alpha=r+1}^{n}\epsilon_{\alpha}\overline{g}(h^{s}(Z_{u},Z_{u}),W_{\alpha})W_{\alpha}.
 \end{align}
 Then using (\ref{h1}) we derive
 \begin{equation}\label{K6}
  \overline{g}(h^{s}(X,Y),W_{\alpha})=\epsilon_{\alpha}h_{\alpha}^{s}(X,Y)=g(A_{W_{\alpha}}X,Y),
 \end{equation}
for any $X,Y\in\Gamma(S(TM))$. Finally, replacing (\ref{K6}) in (\ref{K4}) we get 
 \begin{align}\label{K7}
  \mathrm{trace}(h)|_{S(TM)}&=\sum_{t=1}^{2l}\frac{\epsilon_{t}}{n}\sum_{\alpha=r+1}^{n}\epsilon_{\alpha}g(A_{W_{\alpha}}X_{t},X_{t})W_{\alpha}\nonumber\\
  &+\sum_{j=2p+1}^{\kappa}\frac{1}{n}\sum_{\alpha=r+1}^{n}\epsilon_{\alpha}g(A_{W_{\alpha}}\overline{\phi}E_{j},\overline{\phi}E_{j})W_{\alpha}\nonumber\\
  &+\sum_{j=2p+1}^{\kappa}\frac{1}{n}\sum_{\alpha=r+1}^{n}\epsilon_{\alpha}g(A_{W_{\alpha}}\overline{\phi}N_{j},\overline{\phi}N_{j})W_{\alpha}\nonumber\\
  &+\sum_{d=r+1}^{k}\frac{\epsilon_{d}}{n}\sum_{\alpha=r+1}^{n}\epsilon_{\alpha}g(A_{W_{\alpha}}\overline{\phi}W_{d},\overline{\phi} W_{d})W_{\alpha}\nonumber\\
 &+\sum_{u=\kappa+1}^{r}\frac{\epsilon_{u}}{n}\sum_{\alpha=r+1}^{n}\epsilon_{\alpha}\overline{g}(A_{W_{\alpha}}Z_{u},Z_{u})W_{\alpha},
 \end{align}
 from which our assertion follows. Hence the proof.
\end{proof}
 
\begin{example}
{\rm
 Let $(M,g)$ be a submanifold of $\mathbb{R}_{4}^{2m+1}$ given in Example \ref{exa1}. We have shown that $h^{l}(X,Y)=0$ for any $X,Y\in\Gamma(TM)$. Hence, from (\ref{metric}) we can see that the induced connection $\nabla$ is a metric connection. Further, we have also seen that $h(X,Y)=0$ for all $X,Y\in\Gamma(\mathrm{Rad} \,TM)$ and thus, $h^{s}(X,Y)=0$ on $\mathrm{Rad} \,TM$ and also $h^{s}(X,E)=0$ for all $X\in\Gamma(TM)$. Therefore, $M$ is an irrotational minimal ascreen QGCR-lightlike submanifold of $\mathbb{R}_{4}^{2m+1}$ with $\mathrm{trace}(A_{W_{\alpha}})|_{S(TM)}=0$ and thus  satisfying the above theorem.
 }
\end{example}
\begin{corollary}
  Let $(M,g)$ be a totally umbilical irrotational ascreen QGCR-lightlike submanifold of an indefinite nearly $\mu$-Sasakian manifold $(\overline{M},\overline{g})$. If $\nabla$  is a metric connection, then $M$ is minimal if the mean curvature vectors  $\mathcal{H}^{s}=0$.
\end{corollary}
 \section{Co-screen QGCR-lightlike submanifolds}\label{Co}
In this section, we  study a special class of QGCR-lightlike submanifolds of indefinite nearly $\mu$-Sasakian manifolds, called \textit{co-screen QGCR-lightlike submanifold}.  

\begin{definition}\label{coc}
 {\rm
 Let $(M,g)$ be a QGCR-lightlike submanifold of an indefinte almost contact manifold $(\overline{M},\overline{g})$. We say that $M$ is a co-screen QGCR-lightlike submanifold of $\overline{M}$ if $\xi\in\Gamma(S(TM^{\perp}))$. 
 }
\end{definition}
From  Definition \ref{def2} of QGCR-lightlike submanifold we notice that if $M$ is a co-screen QGCR-lightlike submanifold then the direct sum in (\ref{s81}) reduces to the orthogonal sum $\overline{\phi} \,\overline{D}=  \mathcal{S}\perp \mathcal{L}$. Note that this condition is also satisfied by GCR-lightlike submanifols though $\xi\in\Gamma(S(TM))$. 

In the case of co-screen QGCR, the tangent bundle of $M$ is decomposed as follows; 
\begin{equation}
 TM=D\oplus\overline{D},\;\; \mbox{with}\;\; D=\mathrm{Rad}\, TM\perp D_{0}\perp\overline{\phi}D_{2}.
\end{equation}
The transversal bundle can also be decomposed as
\begin{equation}
\mathrm{tr}(TM)= \overline{\phi}\, \overline{D}\perp \mathcal{G}\perp \mathbb{R}\xi, 
\end{equation}
where $\mathcal{G}$ is an $\overline{\phi}$-invariant distribution, i.e., $\overline{\phi}\mathcal{G}=\mathcal{G}$, and $\mathbb{R}\xi$ is  line bundle spanned by $\xi$.

 Next, we construct an example of a co-screen QGCR-lightlike submanifold of a special nearly $\mu$-Sasakian manifold   $\overline{M}$ in which  $\overline{H}=0$ and $\mu=1$. More precisely, we take $\overline{M}$ to be a Sasakian manifold with the usual contact structure given in \cite{ds2}.
 
 \begin{example}\label{exa11}
{\rm
Let $\overline{M}=(\mathbb{R}_4^{13}, \overline{g})$ be a semi-Euclidean space,  where $\overline{g}$ is of signature $(-,-,+,+,+,+, - ,  -,+,+,+,+,+)$ with respect to the canonical basis
\begin{equation*}
 (\partial x_{1},\partial x_{2},\partial x_{3},\partial x_{4},\partial x_{5},\partial x_{6},\partial y_{1},\partial y_{2},\partial y_{3},\partial y_{4},\partial y_{5},\partial y_{6},\partial z).
\end{equation*}
Let $(M,g)$ be a submanifold of $\overline{M}$ given by 
\begin{equation*}
 x^{1}=y^{4},\;\; y^{1}=-x^{4},\;\; x^{2}= y^{3},\;\; y^5=(x^5)^{\frac{1}{2}}, \;\;\mbox{and}\;\;z=constant.
 \end{equation*}
 By direct calculations, we can easily check that the vector fields
\begin{align*} 
 E_{1} & = \partial x_{4}+\partial y_{1}+y^4\partial z,\;\;\;\; E_{2}=\partial x_{1}-\partial y_{4}+y^1\partial z,\\
 E_{3} & = \partial x_{2} + \partial y_{3} + y^{2} \partial z, \;\;\;\; X_{1}=2y^{5}\partial x_{5}+\partial y_{5}+2(y^5)^{2}\partial z,\\ 
 X_{2} & = \partial x_{3} - \partial y_{2}+y^{3}\partial z,\;\;\;\;  X_{3}=2(\partial x_{3} + \partial y_{2} + y^{3} \partial z),\\
 X_{4} & = 2\partial y_{6}\;\;\;\;\mbox{and}\;\;\;\; X_{5}=2(\partial x_{6}+y^{6}\partial z),
\end{align*}
form a local frame of $TM$. From this, we can see that $\mathrm{Rad} \, TM$ is spanned by $\{E_{1}, E_{2}, E_{3}\}$, and therefore, $M$ is 3-lightlike. Further, $\overline{\phi}_{0} E_{1}=E_{2}$, therefore we set $D_{1}=\mbox{span}\{E_{1},E_2\}$. Also $\overline{\phi}_{0} E_{3}=X_{2}$ and thus $D_{2}=\mathrm{span}\{E_{3}\}$. It is easy to see that $\overline{\phi}_{0} X_{4}=X_{5}$, so we set $D_{0}=\mathrm{span}\{X_{4},X_{5}\}$. On the other hand, following direct calculations, we have 
\begin{align} 
 N_{1} & =2(\partial x_{4}-\partial y_{1}+y^{4}\partial z),\;\; \;  N_{2}=2(-\partial x_{1}-\partial y_{4}+y^{1}\partial z),\nonumber\\
 N_{3} & =2(- \partial x_{2} + \partial y_{3} + y^{2} \partial z),\;\;  W_{1}=\partial x_{5}-2y^{5}\partial y_{5}+y^{5}\partial z,\nonumber\\
 \mbox{and}&\;\; W_{2}=2\partial z, \nonumber
\end{align}
from which $l\mathrm{tr}(TM)=\mathrm{span}\{N_{1},N_{2},N_{3}\}$ and $S(TM^{\perp})=\mathrm{span}\{W_{1}, W_{2}\}$. Clearly, $\overline{\phi}_{0} N_{2}=-N_{1}$. Further, $\overline{\phi}_{0} N_{3}= X_{3}$ and thus $\mathcal{L}=\mbox{Span}\{N_3\}$. Also, $\overline{\phi}_{0} W_{1}=-X_{1}$ and therefore $\mathcal{S}=\mbox{span}\{W_{1}\}$. Notice that, $\overline{\phi}_{0}W_{2}=0$ and $\overline{g}(W_{2},W_{2})=1$ hence $\xi=W_{2}$. Clearly, $M$ is a co-screen QGCR-lightlike submanifold of $\overline{M}$ satisfying the hypothesis of Definition \ref{coc}.
}
\end{example}
Let $(M,g)$ be a co-screen QGCR-lightlike submanifold of an indefinite nearly $\mu$-Sasakian manifold, $(\overline{M},\overline{g})$, and let $S$ and $R$ be the projections of $TM$ on to $D$ and $\overline{D}$ respectively, while $F$ and $Q$ are projections of $\mathrm{tr}(TM)$ on to $\overline{\phi}\, \overline{D}$ and $\mathcal{G}$ respectively. Then, 
\begin{equation}\label{mas28}
 X=SX+RX\;\;\mbox{and}\;\;V=FV+QV+\eta(V)\xi,
\end{equation}
for any $X\in\Gamma(TM)$ and $V\in\Gamma(\mathrm{tr}(TM))$.

Applying $\overline{\phi}$ to the two equations of (\ref{mas28}), we respectively derive 
\begin{equation}\label{mas29}
 \overline{\phi}X=\phi_{1}X+\varphi_{1}X\;\;\mbox{and}\;\;\overline{\phi}V=\phi_{2} V+\varphi_{2}V,
\end{equation}
where $\{\phi_{1}X, \phi_{2} V\}$  and $\{\varphi_{1}X,\varphi_{2}V\}$ respectively belongs to $TM$ and $\mathrm{tr}(TM)$.

Using the nearly $\mu$-Sasakian condition (\ref{eqz}) and equations (\ref{mas29}) and (\ref{eq11})-(\ref{eq32}), we derive 
\begin{align}\label{mas36}
 &-A_{\varphi_{1}X}Y-A_{\varphi_{1}Y}X+\nabla_{X}\phi_{1}Y+\nabla_{Y}\phi_{1}X\nonumber\\
 &+\nabla^{t}_{X}\varphi_{1}Y+\nabla^{t}_{Y}\varphi_{1}X+ h(X,\phi_{1}Y)+h(Y,\phi_{1}X)\nonumber\\
 &-\phi_{1}\nabla_{X}Y-\phi_{1}\nabla_{X}Y-\varphi_{1}\nabla_{X}Y-\varphi_{1}\nabla_{X}Y\\
 &-2\phi_{2}h(X,Y)-2\varphi_{2} h(X,Y)-2\mu\overline{g}(X,Y)\xi=0,\nonumber
\end{align}
for all $X,Y\in\Gamma(TM)$.
Then, comparing the tangential and transversal components in (\ref{mas36}), we get;
\begin{align}\label{mas37}
&\nabla_{X}\phi_{1}Y+\nabla_{Y}\phi_{1}X-A_{\varphi_{1}X}Y-A_{\varphi_{1}Y}X\nonumber\\
&-\phi_{1}\nabla_{X}Y-\phi_{1}\nabla_{X}Y-2\phi_{2}h(X,Y)=0,
\end{align}
 and       
\begin{align}\label{mas38}
&\nabla^{t}_{X}\varphi_{1}Y+\nabla^{t}_{Y}\varphi_{1}X+ h(X,\phi_{1}Y)+h(Y,\phi_{1}X)\nonumber \\ 
-\varphi_{1}&\nabla_{X}Y-\varphi_{1}\nabla_{X}Y-2\varphi_{2}h(X,Y)-2\mu\overline{g}(X,Y)\xi=0,
\end{align}
respectively.
\begin{theorem}
  Let $(M,g)$ be a co-screen QGCR-lightlike submanifold of an indefinite nearly $\mu$-Sasakian manifold $\overline{M}$. Then, 
  \begin{enumerate}
   \item $D$ is integrable if and only if 
       \begin{equation}\nonumber
         h(X,\phi_{1}Y)+h(Y,\phi_{1}X)=2\varphi_{1}\nabla_{Y}X+2\varphi_{2} h(X,Y)+2\mu\overline{g}(X,Y)\xi,
       \end{equation}
   for all $X,Y\in\Gamma(D)$.
  
  \item $\overline{D}$ is integrable if and only if 
  \begin{equation}\nonumber
      A_{\varphi_{1}X}Y+A_{\varphi_{1}Y}X=-2\phi_{1}\nabla_{Y}X-2\phi_{2}h(X,Y),
  \end{equation}
  for all $X,Y\in\Gamma(\overline{D})$.
  \end{enumerate}
\end{theorem}
\begin{proof}
 Using (\ref{mas37}) and (\ref{mas38}), we derive 
 \begin{align}\label{mas40}
         h(X,\phi_{1}Y)+h(Y,\phi_{1}X)&=\varphi_{1}[X,Y]+2\varphi_{1}\nabla_{Y}X+2\varphi_{2} h(X,Y)\nonumber\\
         &+2\mu\overline{g}(X,Y)\xi,
       \end{align}
 for all $X,Y\in\Gamma(D)$ and 
 \begin{equation}\label{mas41}
        A_{\varphi_{1}X}Y+A_{\varphi_{1}Y}X=-\phi_{1}[X,Y]-2\phi_{1}\nabla_{Y}X-2\phi_{2}h(X,Y),
  \end{equation}
 for all $X,Y\in\Gamma(\overline{D})$. Then, the assertions follows from (\ref{mas40}) and (\ref{mas41}), which completes the proof.

\end{proof}
Next, we define nearly parallel distributions of submanifolds in semi-Riemannian manifolds.
\begin{definition}\label{defn3}
 Let $(M,g)$ be a submanifold of a semi-Riemannian manifold $(\overline{M},\overline{g})$ and let $\nabla$  be the connection induced in its tangent bundle. Then a distribution $D$ on $M$ will be called nearly parallel if 
       \begin{equation}\nonumber
       \nabla_{X}Y+\nabla_{Y}X\in\Gamma(D),\;\; \forall \,X\in\Gamma(TM) \;\;\mbox{and}\;\; Y\in\Gamma(D).                                                                                                                                                                                                                                                                                                                                                                                                                                           \end{equation}
 \end{definition}
 \begin{lemma}\label{lemma5}
  Let $(M,g)$ be a co-screen QGCR-lightlike submanifold of an indefinite nearly Sasakian manifold $\overline{M}$. Then
  \begin{equation}\nonumber
   \eta(\overline{\nabla}_{X}Y)+\eta(\overline{\nabla}_{Y}X)=0,
  \end{equation}
 for all $X,Y\in\Gamma(TM)$.
 \end{lemma}
\begin{proof}
 The proof follows by straightforward calculations.
\end{proof}

Now, using Definition \ref{defn3} and Lemma \ref{lemma5}, we have the following;
\begin{theorem}\label{theorem2}
  Let $(M,g)$ be a co-screen QGCR-lightlike submanifold of an indefinite nearly $\mu$-Sasakian manifold $\overline{M}$. If  $D$ is nearly parallel, then $h(X,\overline{\phi}Y)+h(Y,\phi_{1}X)+\nabla^{t}_{Y}\varphi_{1}X$ has no component in $(\mathcal{L}\perp\mathcal{S})$ for all $Y\in\Gamma(D)$ and $X\in\Gamma(TM)$.
\end{theorem}
\begin{proof}
 Using (\ref{eqz}), (\ref{mas29}) and Lemma \ref{lemma5}, we derive
 \begin{align}\label{mas44}
 &-\phi_{1}A_{\varphi_{1}X}Y-\varphi_{1}A_{\varphi_{1}X}Y+\phi_{1}\nabla_{X}\overline{\phi}Y+\varphi_{1}\nabla_{X}\overline{\phi}Y\nonumber\\
 &+\phi_{1}\nabla_{Y}\phi_{1}X+\varphi_{1}\nabla_{Y}\phi_{1}X+\nabla_{Y}X+\nabla_{X}Y+\phi_{2}\nabla_{Y}^{t}\varphi_{1}X\nonumber\\
 &+\varphi_{2}\nabla_{Y}^{t}\varphi_{1}X+2h(X,Y)+\phi_{2}h(X,\overline{\phi}Y)+\varphi_{2}h(X,\overline{\phi}Y)\nonumber\\
 &+\phi_{2}h(\phi_{1}X,Y)+\varphi_{2}h(\phi_{1}X,Y)=0,
 \end{align}
for all $Y\in\Gamma(D)$ and $X\in\Gamma(TM)$. Then, taking the tangential components of (\ref{mas44}), we get 
\begin{align}\label{mas45}
 &-\phi_{1}A_{\varphi_{1}X}Y+\phi_{1}\nabla_{X}\overline{\phi}Y+\phi_{1}\nabla_{Y}\phi_{1}X+\phi_{2}\nabla_{Y}^{t}\varphi_{1}X\nonumber\\
 &+\nabla_{Y}X+\nabla_{X}Y+\phi_{2}h(X,\overline{\phi}Y)+\phi_{2}h(\phi_{1}X,Y)=0,
\end{align}
for all $Y\in\Gamma(D)$ and $X\in\Gamma(TM)$. The result follows from (\ref{mas45}), using the fact that $D$ is nearly parallel.
\end{proof}
\begin{theorem}
  Let $(M,g)$ be a co-screen QGCR-lightlike submanifold of an indefinite nearly $\mu$-Sasakian manifold $\overline{M}$. If  $\overline{D}$ is nearly parallel, then $-A_{\varphi_{1}X}Y-A_{\overline{\phi}Y}X+\nabla_{Y}\phi_{1}X$ has no component in $D$ for all $Y\in\Gamma(\overline{D})$ and $X\in\Gamma(TM)$.
\end{theorem}
\begin{proof}
 The proof is similar to that in Theorem \ref{theorem2} and therefore we leave it out.
\end{proof}
Using the idea of \cite{MM}, we define nearly auto-parallel distributons on submanifolds of semi-Riemannian manifolds.
\begin{definition}\label{defn4}
 Let $(M,g)$ be a submanifold of a semi-Riemannian manifold $(\overline{M},\overline{g})$ and let $\nabla$  be the connection induced in its tangent bundle. Then a distribution $D$ on $M$ will be called nearly auto-parallel if 
       \begin{equation}\nonumber
       \nabla_{X}Y+\nabla_{Y}X\in\Gamma(D),\;\; \forall \, X,\,Y\in\Gamma(D).                                                                                                                                                                                                                                                                                                                                                                                                                                           \end{equation}
 \end{definition}
\begin{theorem}\label{theorem3}
  Let $(M,g)$ be a co-screen QGCR-lightlike submanifold of an indefinite nearly $\mu$-Sasakian manifold $\overline{M}$. If  $D$ is nearly auto-parallel, then $h(X,\overline{\phi}Y)+h(Y,\phi_{1}X)$ has no component in $(\mathcal{L}\perp\mathcal{S})$ for all $X,Y\in\Gamma(D)$.
\end{theorem}

\begin{proof}
  Using (\ref{eqz}), equations (\ref{mas29}) and Lemma \ref{lemma5}, we derive
  \begin{align}\label{mas47}
   &\phi_{1}\nabla_{X}\overline{\phi}Y+ \varphi_{1}\nabla_{X}\overline{\phi}Y+\phi_{1}\nabla_{Y}\overline{\phi}X+ \varphi_{1}\nabla_{Y}\overline{\phi}X\nonumber\\
   +\phi_{2}h&(X,\overline{\phi}Y)+\varphi_{2}h(X,\overline{\phi}Y)+\phi_{2}h(Y,\overline{\phi}X)+\varphi_{2}h(Y,\overline{\phi}X)\nonumber\\
   &+\nabla_{X}Y+\nabla_{Y}X+2h(X,Y)=0,\;\;\forall\,X,Y\in\Gamma(D).
  \end{align}
  Considering the tangential components of (\ref{mas47}) we get 
  \begin{align}\label{mas48}
   \phi_{1}\nabla_{X}&\overline{\phi}Y+\phi_{1}\nabla_{Y}\overline{\phi}X+\phi_{2}h(X,\overline{\phi}Y)+\phi_{2}h(Y,\overline{\phi}X)\nonumber\\
   &+\nabla_{X}Y+\nabla_{Y}X=0,\;\;\forall\,X,Y\in\Gamma(D). 
  \end{align}
 Since $D$ is nearly auto-parallel, (\ref{mas48}) leads to 
 \begin{equation}\nonumber
  \phi_{2}h(X,\overline{\phi}Y)+\phi_{2}h(Y,\overline{\phi}X)=0,
 \end{equation}
from which our assertion follows. Hence,  the proof is complete.
\end{proof}
In the similar way, we have the following.
\begin{theorem}
  Let $(M,g)$ be a co-screen QGCR-lightlike submanifold of an indefinite nearly $\mu$-Sasakian manifold $\overline{M}$. If  $\overline{D}$ is nearly auto-parallel, then $A_{\varphi_{1}X}Y+A_{\overline{\phi}Y}X$ has no component in $D$ for all $X,Y\in\Gamma(\overline{D})$.
\end{theorem} 
\section*{Acknowledgments}
This research work was partially supported by the African Institute for mathematical Sciences (AIMS) in S\'en\'egal, University of KwaZulu-Natal in South Africa and the Simon Foundation through the RGSM-Network project.

\end{document}